\documentclass{article} \usepackage{amssymb}\usepackage{amscd}
\usepackage{a4wide}
\usepackage{amstext}
\usepackage{amsthm} \newtheorem{theorem}{Theorem} \newtheorem{lemma}{Lemma}
 \newtheorem{remark}{Remark}\newtheorem{definition}{Definition}
 \newcommand{\La}{\Lambda}
\newcommand{\R}{{\mathbb R}}  \newcommand{\Z}{{\mathbb Z}} \newcommand{\N}{{\mathbb N}}
\newcommand{\T}{{\mathbb T}}

\begin{document}

\title{Discrete Uniqueness Sets for  Functions with  Spectral Gaps}

\author{Alexander Olevskii  and  Alexander Ulanovskii}

\date{} \maketitle

\noindent A.O.: School of Mathematics, Tel Aviv University, Ramat Aviv,  69978 Israel\\ E-mail:
 olevskii@post.tau.ac.il

\medskip

\noindent A.U.: Stavanger University,  4036 Stavanger, Norway\\ E-mail: Alexander.Ulanovskii@uis.no

\begin{abstract}
  It is well-known that entire functions whose spectrum belongs to a  fixed bounded set $S$,  admit real uniformly discrete uniqueness sets $\La$. We show that the same is true for much wider spaces of continuous functions. In particular,   Sobolev spaces have this property whenever $S$ is a set of infinite measure having "periodic gaps".
   The periodicity condition is crucial. For  sets $S$ with randomly distributed gaps, we show that the uniformly discrete sets $\La$  satisfy a strong non-uniqueness property: Every discrete function $c(\lambda)\in l^2(\La)$ can be interpolated by an analytic $L^2$-function with spectrum in $S$.\end{abstract}
 \begin{large}

\section{Introduction}


\noindent{\bf  Paley-Wiener Space}.
We will use the standard form of the
 Fourier transform:
$$
F(t)=\hat f(t):=\int_\R e^{-2\pi i tx}f(x)\,dx.
$$

    Given a measurable set $S$, the Paley-Wiener space $PW_S$ consists of the inverse Fourier transforms of all square-integrable functions  $F$ which vanish a.e. outside $S$. The set $S$ is called the spectrum of the  space $PW_S$. Clearly, if the measure of $S$ is finite, then $F\in L^1(\R)$, and so every function $f \in PW_S$ is continuous. If $S$ is bounded, then every  $f \in PW_S$ is an entire function of exponential type.

\medskip
\noindent{\bf Uniformly discrete sets}. A set $\La\subset\R$ is called {\it uniformly discrete} (u.d.) if $\delta(\La)>0$, where
\begin{equation}\label{sep}\delta(\La):=\inf_{\lambda,\lambda'\in\La,\lambda\ne\lambda'} |\lambda-\lambda'|\end{equation}
is the infimal distance between different elements of $\La$.

 A u.d. set $\La$ is said to have the {\it uniform density} $D(\La)$, if $\La$ is regularly distributed in the following sense:$$\mbox{Card}\left(
\La\cap(x,x+r)\right)=rD(\La)+o(r) \mbox{ uniformly on } x \mbox{ as } r\to \infty.$$

\medskip
\noindent{\bf Uniqueness Problem}.
Let $F$ be a space of continuous functions on the real line $\R$. A set $\La\subset\R$ is called a {\it uniqueness set}  for $F$ if
$$f \in F, f|_\La =0 \Rightarrow f =0.$$
Otherwise, $\La$ is called a {\it non-uniqueness set}.

\medskip\noindent{\bf Problem}. {\it Which  spaces of continuous functions on $\R$  admit u.d. uniqueness sets?}

\medskip

We will consider this problem for spaces of function whose spectrum belongs to a fixed set $S$.
It is natural to  distinguish between the following cases: $S$  is a bounded set, an unbounded set of finite measure, and a set of infinite measure.

In the present paper we focus on spaces of continuous functions whose spectrum lies in a set  $S$ of  infinite measure. In Sec. 3--5 we establish that wide spaces of such functions admit u.d. uniqueness sets, provided  $S$ has periodic gaps.
The periodicity condition  is important. In particular, in Sec. 6, for sets $S$ with randomly distributed gaps we show that every u.d. set $\La$ satisfies some strong non-uniqueness property.

We start with a short survey of known results on the first two cases. A detailed discussion of these and related results can be found in \cite{ou1}. For simplicity of presentation, we focus on the one-dimensional case.

\section{Spectra of Finite Measure}

\noindent{\bf Bounded Spectra}.
The classical  case is when $S = [a,b]$  an interval. Then the elements of $PW_S$ are entire functions of exponential type. The distribution of zeros of such functions is very well studied, see
\cite{levin}. In particular, if the density $D(\La)$ exists, then the condition $D(\La)\geq |S|$ is necessary while the condition
$D(\La) >|S|$
is sufficient for $\La$ to be a uniqueness set for $PW_S$, where $|S|=b-a$ denotes the measure of $S$. This can be shown by standard complex variable techniques. A classical result of Beurling and Malliavin \cite{BM} states that the same is true for irregular sets $\La$, provided the uniform density is replaced with a certain exterior one (the Beurling--Malliavin density).

In the case of disconnected spectra $S$, the uniqueness property of u.d. sets cannot be expressed in terms of their density: Some "dense" (relatively to the measure of $S$) u.d. sets $\La$ may be  non-uniqueness sets for $PW_S$.  For example, one can easily check that $\La=\Z$ is a non-uniqueness set
for $PW_S$, where $S = [0,\epsilon]\cup[1,1+\epsilon],0< \epsilon<1$.

On the other hand,  some "sparse" u.d. sets $\La$ may be uniqueness sets for $PW_S$ with a "large" spectrum $S$. This phenomenon was  discovered by Landau \cite{L}, who proved that certain perturbations of $\Z$ produce uniqueness sets for $PW_S$ whenever $S$ is a finite union of intervals $[k+a,k+1-a],$ where $0<a<1/2$ is any fixed number. The uniqueness sets $\La$  constructed by Landau have a complicated structure.

 A more general result is proved in   \cite{ou2}:

 \begin{theorem} The set $$
 \La:=\{n+2^{-|n|}, n\in\Z\}
 $$ is a uniqueness set for $PW_S$, for every bounded set $S$ satisfying $|S|<1.$
 \end{theorem}

This theorem remains true for the bounded sets $S$ of arbitrarily large measure satisfying $|S_1|<1$, where we denote by
\begin{equation}\label{pr}
S_a:=(S+a\Z)\cap[0,a]
\end{equation} the "projection" of $S$ onto $[0,a]$.

Moreover, the result holds true also for the unbounded sets of finite measure which have  a "moderate accumulation" at infinity, see \cite{ou2}.

Using re-scaling, one may formulate a corresponding result for any  bounded set $S$.

\medskip\noindent{\bf Unbounded Spectra of Finite Measure}.
It was shown in  \cite{ou} (see also \cite{ou1}, Lec. 10)  that for every (bounded or unbounded) set $S$ in $\R$ of finite measure, the space $PW_S$ possesses a u. d. uniqueness set:

\begin{theorem} For every set $S$ of finite measure, there is  a u.d. set $\La$ satisfying  $D(\La)=|S|$, which is a uniqueness set for $PW_S$.\end{theorem}

By the discussion above, the density condition $D(\La)=|S|$ is optimal, since one cannot get a smaller  density when $S$ is an interval.

\section{Sobolev Spaces with Periodic  Spectral Gaps}

\subsection{Periodic Spectral Gaps}

 We say that $S$ has    periodic  "strong" gaps if there exists $a>0$ such that \begin{equation}\label{2}
|\overline{S_a}|<a,
\end{equation}
 where $\overline{S_a}$ denotes the closure of $S_a$, and the set $S_a$ is defined in (\ref{pr}).
Condition (\ref{2}) means that there is a non-empty interval $I\subset[0,a]$ such that $S\cap (I+a\Z)=\emptyset.$

We say that $S$ has  periodic "weak"  gaps if
 \begin{equation}\label{3}
|S_a|<a.
\end{equation}Condition (\ref{3}) means that there is a set of positive  measure  $Q\subset[0,a]$ such that
$S\cap (Q+a\Z)=\emptyset.$

Observe that {\it every set $S$ of finite measure has  periodic weak gaps}, since we have $|S_a|<a$, for every $a>|S|.$

\subsection{Uniqueness Sets for Sobolev Spaces}

Given any u.d. set $\La$, it is obvious that there is a non-trivial smooth function $f$ which vanishes on $\La$. However, this is no longer so if the spectrum of $f$  has  weak periodic gaps.
We will state the result for Sobolev spaces.

For every number $\alpha>1/2$, we denote by $W^{(\alpha)}$ the Sobolev space of  functions  $f$ such that the Fourier transform $F=\hat f$  satisfies
\begin{equation}\label{4}
\|F\|_\alpha^2:=\int_\R (1+|t|^{2\alpha})|F(t)|^2\,dt<\infty.
\end{equation}
It is clear that the functions $F$ satisfying (\ref{4}) belong to $L^1(\R)$, and so $W^{(\alpha)}$ consists of continuous functions.
 We denote by $W^{(\alpha)}_S$ the subspace of $W^{(\alpha)}$ of functions $f$ with spectrum in $S$, i.e.  $F=0$  a.e. outside $S$.

\begin{theorem}\label{t1}
Suppose a set $S$ satisfies $|S_a|<a$, for some $a>0$. Then there is a u.d. set $\La$ of density $D(\La)=a$, which is a uniqueness set for  $W^{(\alpha)}_{S}, \alpha>1/2$.
\end{theorem}

\subsection{Decomposition of $\Z$}

\begin{lemma} Let  $A \subset[0,1], |A| < 1$. There exist pairwise disjoint sets $Z_j \subset\Z, j\in\N$, such that
every exponential system \begin{equation}\label{exps}\{e^{-i 2\pi n t},n\in Z_j\}\end{equation} is complete in $L^2(A)$.
\end{lemma}

\begin{proof}
1. Observe that the exponential family
\begin{equation}\label{sys}\{e^{-i2\pi nt}, |n| > N\} \end{equation}is complete in $L^2(A)$,  for every natural $N$. Indeed, assume there exists a non-trivial function  $F\in L^2(A)$   orthogonal  to the system (\ref{sys}). Extend $F$ by zero to $[0,1]\setminus A$. Then $F$ is orthogonal  to the system (\ref{sys}) in $L^2(0,1)$.
 Since the trigonometrical system forms an orthonormal basis in $L^2(0,1)$, we conclude that $F$ is a trigonometric polynomial:
$$
F(t)=\sum_{|n|\leq N}c_je^{-i2\pi n t}.
$$Clearly, $F$ cannot vanish on the set  of positive measure $[0,1]\setminus A$, which is a contradiction.

2. Fix a sequence $\epsilon_k, k\in\N$, satisfying $\epsilon_k\to0, k\to\infty.$ We will now construct a sequence of disjoint finite symmetric sets $\Gamma_k\subset\Z, k\in \N,$ with the following property:
For every $|m|\leq k,$ there is a trigonometric polynomial $P_{k,m}$ whose frequencies belong to $\Gamma_k$, such that
\begin{equation}\label{s}
\|e^{i2\pi m t}-P_{k,m}(t)\|_{L^2(A)}<\epsilon_k.
\end{equation}

Set $\Gamma_1:=\{-1,0,1\}$.  Clearly, (\ref{s}) holds with $m=0,-1,1$.
Then set$$
\Gamma_k:=\{n:n_{k-1}<|n|\leq n_k\},
$$where  $n_1=1$, and we choose $n_j, j>1,$ inductively as follows:
By Step 1, there exists $n_2$ so large that for every $|m|\leq 2$ there is a polynomial
$P_{2,m}$ satisfying  (\ref{s}) with $k=2$ and whose frequencies belong to the set $\Gamma_2$,
and so on. On the $k$-th step, we choose an integer $n_k$ so large that for every $|m|\leq k$ there is a polynomial
$P_{k,m}$ satisfying  (\ref{s}) and whose frequencies belong to the set $\Gamma_k$.

3. Now, take a partition of $\N$ into disjoint infinite subsets $\Delta_j$ and set
$$
Z_j:=\bigcup_{k\in \Delta_j}\Gamma_k.
$$
It follows from the construction above that every exponential system (\ref{exps}) 
is complete in $L^2(A)$.
\end{proof}

\begin{remark} It is easy to see that  one may construct the sets $Z_j$ so that  $$D(\cup_{j=1}^\infty Z_j)=1.$$
\end{remark}

\subsection{Periodization and Fourier Transform}

              For an integrable function $H$ on the circle group $\T:=\R/\Z$,
               we denote by$$
               c_n(H) := \int_\T H(t) e^{2\pi i nt}\, dt,\quad n\in\Z,
$$            the Fourier coefficients of $H$.

               Given  $F\in L^1 (\R)$, consider
              its "periodization"
$$
                H(u):= \sum_{k\in\Z}  F(u+k), \quad    u\in [0,1].
$$
            Clearly,  $H$ is defined a.e. and belongs to $L^1(\T)$.
              Direct calculation shows that its Fourier coefficients satisfy $                     c_n (H) = f(n),$ where $f$ is the inverse Fourier transform of $F$.

              Similarly, for the periodization $H_v$  of the function
$$                     F_v(t):=   e^{2\pi i vt} F(t),
$$
              we have
\begin{equation}\label{ol}
                    c_n (H_v)= f(n+v), \quad n\in\Z.    \end{equation}

It is easy to check that the periodization of an $L^2$-function does not always belong to $L^2(\T)$.
However, the following is true:

\begin{lemma}
Assume $F$ satisfies $\|F\|_\alpha<\infty.$   Then 
$$
\int_0^1 |H(t)|^2\,dt<\infty.
$$
\end{lemma}

\begin{proof}
Indeed, we have
$$
|H(t)|^2=\left|\sum_{n\in\Z}F(t+na)\right|^2\leq \sum_{n\in\Z}\frac{1}{(1+|n|^\alpha)^2}\sum_{n\in\Z}|F(t+na)|^2(1+|n|^\alpha)^2,
$$and the lemma easily follows from the definition of $\|F\|_\alpha$ in (\ref{4}).
\end{proof}

\subsection{Proof of Theorem 3}

              By re-scaling we can assume that $a=1$.

            Using Lemma 1 with $A=S_1$, write $\Z= \bigcup_{j=1}^\infty Z_j$, where each exponential system (\ref{exps}) is  complete in $L^2(S_1)$. It means that each $Z_j$ is a uniqueness
set for the space $PW_{S_1}$.

Fix a sequence $\{\alpha_j\}$ dense in $[0,1],$ and set
\begin{equation}\label{ol2}
              \La:= \bigcup_{j=1}^\infty (Z_j + \alpha_j).
\end{equation}We may assume that $\La$ is u.d. and $D(\La)=1.$

            Now we will prove that $\La$ is a uniqueness set for the space $W_S^{(\alpha)}$. We have to show that every function  $f\in W_S^{(\alpha)}$ satisfying
\begin{equation}\label{ol3}
                     f|_\La = 0
\end{equation}
         must vanish on  $\R$.

         Let $F:=\hat f$. Fix $j\in\N$ and  consider the function
$$
               F_j(t):=  e^{2\pi \alpha_j t} F(t),
$$
             and its periodization $H_j$.         Recall that $F$ vanishes a.e. outside $S$. Since $S\subset S_1+\Z$,  we have
$$
                     H_j = 0 \quad \mbox{a.e. on } \T\setminus S_1.
$$                    Also, by Lemma 2, $H_j \in L^2 (\T)$.

          By (\ref{ol}), (\ref{ol2}) and (\ref{ol3}),
$$
               c_n (H_j)= f(\alpha_j +n)  = 0,\quad  n\in Z_j.
$$Since $Z_j$ is a uniqueness
set for $PW_{S_1}$, we have $H_j=0$ a.e.
  By (\ref{ol}), this means that
$$
                  f(n+\alpha_j) = 0,\quad   n\in \Z.
$$
    Since this equality   is true for all $j $, $f$ is continuous and the sequence $\{\alpha_j\}$ is dense in $[0,1]$, we conclude that $ f= 0.$

\section{Uniqueness Sets for  Fast Decreasing Functions}
 Theorem \ref{t1} shows that certain classes of smooth functions $f$ having periodic weak spectral gaps admit u.d. uniqueness sets.
In this section we show that  a similar  result holds for  functions  $f$ whose Fourier transforms $F$ are smooth functions.

Let us denote by $Y$  the space of continuous functions $f$ satisfying
\begin{equation}\label{dec}
\sup_{x\in\R}(1+x^2)|f(x)|<\infty.
\end{equation}
Denote by $Y_S$ the subspace of $Y$ of functions $f$ such that $F=\hat f=0$ outside $S$.

\begin{theorem}\label{t2}
Suppose a set $S$ satisfies $|S_a|<a$, for some $a>0$. Then there is a u.d. set $\La$ of density $D(\La)=a$, which is a uniqueness set for  $Y_{S}$.
\end{theorem}

The proof below  shows that in Theorem \ref{t2} condition (\ref{dec}) in the definition of $Y$ can be somewhat relaxed. However, the result is no longer true if no decay condition is imposed, see Theorem 6 below.

\subsection{Proof of Theorem \ref{t2}}

         The proof follows the same idea used in the proof of  Theorem 3. However,  the periodization of $F$ cannot  be defined pointwisely.

          We will use the following corollary of the classical          Poisson summation formula:

\begin{lemma}\label{l3}  Assume   a continuous function $f$ satisfies (\ref{dec}) and $\hat f(t)=0, t\in Q+\Z$, for some set $Q\subset[0,1], |Q|>0$. Then for every $x\in [0,1]$ we have
\begin{equation}\label{pp1}
\sum_{n\in\Z}f(x+n)e^{-i2\pi nt}=0, \quad t\in Q.
\end{equation}
\end{lemma}

       \begin{proof}   If $F=\hat f$ is also fast decreasing, then  this claim follows
                   directly from the Poisson formula.

                     Otherwise, apply the Poisson formula to the convolution
$(f\ast h_\epsilon)(x)$, where $$h_\epsilon:=\left(\frac{1}{2\epsilon}{\bf 1_{(-\epsilon,\epsilon)}}\right)^{2\ast}$$ and ${\bf 1_{(-\epsilon,\epsilon)}}$ is the indicator function of $(-\epsilon,\epsilon)$:
$$
                  \sum_{n\in\Z} (f*h_\epsilon)(x+n)e^{-i2\pi  nx} = 0, \quad  x\in [0,1),\ t\in Q.
$$

                  Now, we claim:
$$
                 \sum_{n\in\Z} (f\ast h_\epsilon)(x+n)e^{-i2\pi nt}  \to \sum_{n\in\Z}  f(x+n)e^{-i2\pi nt}
  $$               as $\epsilon\to 0$, which proves the lemma.
Indeed, fix $\delta>0$ and decompose the left side into two sums: $
                  \sum_{|n|< N} + \sum_{ |n|\geq N}.$
                  One can chose $N= N(\delta)$ so that modulus of the second
                  summand is $< \delta$, for every $x,t\in[0,1]$ and $0<\epsilon<1$.  Clearly, each  term of the first summand goes to $f(x)e^{-i2\pi nt}$ as $\epsilon\to 0,$
                  due to the continuity of $f$.
\end{proof}

                  Now, we can finish the proof of Theorem 4.
 By re-scaling, we may assume that $a=1$,  so that  $|S_1|<1$.

 Following the proof of Theorem 3, we may find pairwise disjoint sets              $Z_j\subset\Z, j\in\N,$  such that for every $j$ the system
 $$E(Z_l):=\{e^{-2\pi i kt}, k\in Z_l\}$$ is complete in $L^2(S_1)$.

Set
  $$
  \La:=\cup_{j\in\N}(Z_j+\alpha_j),
  $$where  $\{\alpha_l, l\in\N\}$ is dense in $(0,1)$.
             It remains to check that $\La$ is a uniqueness set for $Y_S$.

     Assume $f|_\La = 0$, for some $f\in Y_S$, i.e. we have
$$
                   f|_{Z_j + \alpha_j} = 0, \quad j=1,2,\dots
$$

           Fix $j$ and consider a $1$-periodic function
$$
                g_j (x) :=\sum_{n\in \Z}  f(n + \alpha_j) e^{-i 2\pi nx}.
$$
          Clearly, $g\in L^2(0,1)$ and is orthogonal in $L^2(0,1)$ to all the exponential functions in $E(Z_j)$.   On the other hand, due to Lemma 3,
$$
                    g_j(x) = 0,\quad  t\in Q:=[0,1]\setminus S_1.
$$
      The completeness of $E(Z_j)$ in $L^2(S_1)$ implies that  $g_j = 0$ a.e. Hence,  $$f(n+\alpha_j) = 0,\quad  n\in\Z.$$
          This is true for every $j$.           Recalling that $\{\alpha_j\}$ is dense on $[0,1]$ and  $f \in C(\R) $,
          we conclude that  $f=0$ on $\R$.

\section{Distributions with Periodic Spectral Gaps}

\subsection{Strong Gaps}

           If  $S$ has periodic strong gaps, then
               the results above can be extended  to  wider             function spaces.

       Denote by $X$  the space of continuous functions that have at most polynomial growth on $\R$. Every element  $f\in X$ is a  Schwartz distribution. Its spectrum is the minimal closed set $S$ such that for every test function $\varphi$ satisfying $\hat \varphi=0$ in a neighbourhood of $S$, we have
$$
\int_\R f(t)\varphi(t)\,dt=0.
$$

Given  closed set $S$, we denote by $X_S$ the subspace of $X$ consisting of functions with spectrum in $S$.

Without loss of generality, we may assume       the spectral gaps are $[0,\delta] + \Z.$

\begin{theorem}\label{t3} There is a u.d. set $\La, D(\La) =1 $,   which is a uniqueness set for $X_S,  S = [0, 1-\delta]+\Z$,  for every $0<\delta<1.$
\end{theorem}

\begin{proof}
           Consider $Z_j$ as in Lemma 1 (this can be done independently on $\delta$). Choose  $\La$ as in the proof of Theorem 4.
          Given  $f\in  X_S,$  consider the function                      $ g:= f\cdot\varphi,$
          where $\hat\varphi$ is a Schwartz function supported by $[0 ,\delta/2]$.

          It is easy to see that $g$ satisfies the assumptions          of Theorem 4 with $a=1 ,S= [0,1-\delta/2]$.
           If $f|_\La = 0$, then the same is true for $g$.  So, Theorem 4 implies $f= 0.$\end{proof}

  \subsection{Weak Gaps}

        Here  we show that Theorem 5 is no longer true  for the weak spectral gaps.
       This is a direct corollary of a result from \cite{ou09}.

        We need the following

        \begin{definition} Given a closed (not necessarily bounded) set $S$,  the Bernstein space $B_S$ is the set of continuous bounded functions $f$ on $\R$ whose spectrum (in distributional sense) lies in $S$.
\end{definition}

     \begin{theorem}{\rm (\cite{ou09})}
There is a closed set  $S$ of Lebesque  measure zero  such that every bounded function $c(\lambda)$ defined on a u.d. set $\La$ can be interpolated by a function $f\in B_S$.
\end{theorem}

It is obvious that every set of measure zero has weak periodic gaps with an arbitrary period $a$. However, no u.d. set $\La$ is a uniqueness set for $B_S$.

       A few words about the proof of Theorem 6. It is based on
         a classical result of D.E. Menshov (1916)  (see \cite{b}):
         {\it There is a probability measure $\mu$ on $\R$ supported by a compact
            set $K$ of measure zero, and such that its Fourier transform
$$
\hat\mu (x)=\int_K e^{-2\pi i t}\,d\mu(t)
$$
           vanishes at infinity}.

 Here is a short sketch of the proof (see details in \cite{ou1}, Lec. 10).

\begin{proof}
1. Given $\delta$, by re-scaling one can get a probability measure $\mu_\delta$ supported by a compact $K$ of Lebesgue measure zero, such that $$
\hat\mu_\delta(0)=1, |\hat\mu_\delta(x)|<\delta,\quad |x|>\delta.
$$

2. Using this, one can construct a family of compact sets $K_j$ of measure zero, which goes to infinity, and functions $g_j\in B_{K_j}, j\in\N,$ satisfying
$$\|g_j\|_\infty=g_j(0)=1, \ |g_j(t)|<e^{-j}, \quad  |t|>e^{-j}.$$

3. Set
$$
S:=\cup_{j=1}^\infty K_j.$$It is a closed (non-compact) set of measure zero.

Fix any $\delta>0$ and any u.d. set $\La$. Using appropriate translates of the functions $g_j$, one can define functions $f_\lambda\in B_S$ satisfying
$$
\|f_\lambda\|_\infty=f_j(\lambda)=1, |f(\lambda')|<e^{-2|\lambda'-\lambda|/\delta},\quad \lambda_j\in\La,\lambda_l\ne\lambda_j.
$$

4.
   Consider the linear operator $T:l^\infty(\La) \rightarrow l^\infty(\La) $ defined by $$(Tc)_\lambda:=\sum_{\lambda'\in\La,\lambda'\ne \lambda}f_{\lambda'}(\lambda)c_{\lambda'}, \quad \lambda\in\La, \quad c=\{c_{\lambda'}, \lambda'\in\La\}\in l^\infty(\La).$$
Clearly, $\|T\|<1$.  Hence, the operator $ T+I$ is surjective. Therefore, for every datum $c=\{c_\lambda\}\in l^\infty(\La)$ there is a sequence $ b=\{b_\lambda\}\in l^\infty(\La)$ satisfying
$(I+T)b=c$.
Hence, the function
$$
               f(x):= \sum_{\lambda\in\La} b_\lambda f_\lambda (x).
$$
belongs to $B_S$ and solves the interpolation problem $  f|_\La= c$.
\end{proof}

\section{Non-Periodic Spectral Gaps}

    Here we show that the periodicity of spectral gaps  is crucial for existence of  discrete uniqueness sets.

Let us consider spectra $S$ which are unions of disjoint intervals of a given length. For simplicity, we assume that each interval has length one:
\begin{equation}\label{gam}
S=\cup_{j=1}^\infty [\gamma_j,\gamma_j+1].
\end{equation}We also assume that the distances $\xi_j$ between the  intervals belong to a fixed interval, say $[2,3]$:
 \begin{equation}\label{xi}
  \quad \xi_j:= \gamma_{j+1}-\gamma_j-1\in [2,3], \quad j\in\N.
 \end{equation}
 Clearly, $S$ belongs to a half-line $[\gamma_1,\infty)$ and admits a representation
\begin{equation}\label{g}
S=\Gamma+[0,1], \  \Gamma:=\cup_{j=1}^\infty\{\gamma_j\},
\end{equation}where $\Gamma$ satisfies $\delta(\Gamma)\geq 3$. Here $\delta(\Gamma)$ is the separation constant defined in (1).

Now we introduce a certain property of u.d. sets.  We say that u.d. set $\Gamma$ satisfies property  (C)
    if it contains arbitrary long arithmetic progressions with rationally      independent steps.
       More precisely, we assume

       \medskip

        (C) For every $m\in\N$ there are rationally independent numbers $q_1,...,q_m$, such that for every $N\in\N$ the set   $\Gamma$ contains arithmetic progressions of length $N$ with differences $q_1,...,q_m$.  The latter means that there exist $a_1,...,a_m$ such that
$$
\bigcup_{j=1}^m\{a_j+q_j, a_j+2q_j,\dots,a_j+Nq_j\}\subset\Gamma.
$$

\begin{theorem}\label{int} Assume $S$ is given in  (\ref{gam})--(\ref{g}), where $\Gamma$ satisfies  property {\rm (C)}. 
Then     no u.d. set $\La$ is a uniqueness set for the Sobolev space $W_S^{\alpha}$.\end{theorem}

   Below we prove this result in a stronger form.

One may also check that, under the assumptions of Theorem \ref{int}, no u.d. set $\La$ is a uniqueness set for the space $Y_S$.

\subsection{Interpolation Sets}

A set $\La$ is called an {\it interpolation set} for the                    Paley-Wiener space $PW_S$, if for every sequence $\{c_\lambda,\lambda\in\La\}\in l^2(\La)$ there exists $f\in PW_S$ satisfying$$
f(\lambda)=c_\lambda,\quad \lambda\in\La.
$$

The following criteria is well-known (see e.g. \cite{ou1}, Lec. 4):

\begin{lemma}Let $S$ be a bounded set and $\La$ a u.d. set. Then $\La$ is a set of interpolation for $PW_S$ if and only if there is a constant $C>0$ such that the inequality
$$
\int_S \left|\sum_{\lambda\in\La}c_\lambda e^{i2\pi\lambda t}\right|^2\,dt\geq C \sum_{\lambda\in\La}|c_\lambda|^2
$$holds  for every finite sequence $c_\lambda$.
\end{lemma}

 Theorem 7 is a direct corollary of the following

\medskip\noindent{\bf Main Lemma}. {\it  Assume $S$ is a set from Theorem 7. Then for every $\delta>0$     there is a bounded subset $S(\delta)\subset S$ such that  every u.d. set $\La$ satisfying $\delta(\La)\geq\delta$  is a set of interpolation  for $PW_{S(\delta)}$.}

\medskip       Indeed, consider a set $\La\cup\{c\}$ for some point $c\not\in\La$. By the Main Lemma, it is a set of interpolation for  $PW_{S(\delta)}$, for some bounded subset $S(\delta)\subset S$. Then  there exists $f\in PW_{S(\delta)}$ satisfying $f(c)=1$ and $$ f(\lambda)=0,\quad \lambda\in\La.$$ It remains to observe that $PW_{S(\delta)}\subset W_S^{\alpha}$.

\subsection{Proof of Main Lemma}

\begin{lemma} Suppose $\Gamma$ satisfies property {\rm (C)}. Then for every $0<\epsilon<1$   there exist $N\in\N$ and $\eta_{j}\in\Gamma, j=1,...,N,$ such that the exponential polynomial
$$
P(t)=\frac{1}{N}\sum_{j=1}^N e^{i \eta_{j}t}
$$
satisfies
\begin{equation}\label{p}
|P(t)|\leq\epsilon , \quad \epsilon < |t|<1/\epsilon.
\end{equation}\end{lemma}

\begin{proof}
1. Fix any integer $m>1/\epsilon$. Then fix numbers $q_1,...,q_m$  in the definition of property (C). Since $q_j$ are rationally independent,  the set of points \begin{equation}\label{set}\{kq_j\in(-1/\epsilon,1/\epsilon)\}\end{equation} is separated, where $ k\in\Z,k\ne0, j=1,\dots, m$. Hence,  the distance between any two points in this set exceeds some positive number $\rho$. We may assume that $\rho<\epsilon$.

2. For $n\in\N$ and $q\geq 2$, consider the $(1/q)$-periodic exponential polynomial
 $$P_{n,q}(t):=\frac{1}{n}\sum_{j=0}^{n-1}e^{i 2\pi j q t}=\frac{1}{n}\frac{e^ {i2\pi nq t}-1}{e^{i2\pi q t}-1}.$$
From the properties of Dirichlet kernel, it is well-known that it satisfies
\begin{equation}\label{rho}
|P_{n,q}(t)|\leq\rho, \quad \mbox{dist}(t, (1/q)\Z)\geq\rho,
\end{equation}provided $n$ is large enough.

3.
Choose $n$ so large that (\ref{rho}) holds with $q=q_j$,  $j=1,\dots,m$. Then, since the set (\ref{set}) is $\rho$-separated, for every $t$ satisfying $\epsilon<|t|<1/\epsilon,$ the inequality
$$
|P_{n, q_j}(t)|<\epsilon
$$holds for all but at most one value of $j\in\{1,\dots,m\}$.

4. By the definition of property (C), there exist $a_j$  such that $a_j+kq_j\in\Gamma, k=0,\dots, n-1$. Set
$$P(t)=\frac{1}{m}\sum_{j=1}^m e^{i2\pi a_j t}P_{n, q_j}(t).$$
By Step 3, we see that
$$
|P(t)|<\frac{1+(m-1)\epsilon}{m}<\epsilon, \quad \epsilon<|t|<1/\epsilon,
$$which completes the proof.\end{proof}

\begin{proof}[Proof of Main Lemma] Fix $\delta>0$ and assume that a u.d. set $\La$ satisfies $\delta(\La)\geq\delta$.

By Lemma 5, for every   $0<\epsilon<1$ there is an exponential polynomial $P$ with frequencies in $\Gamma$ satisfying (\ref{p}). We denote the set of its frequencies   by $\Gamma_P\subset\Gamma$, and set $$S(\delta):=\Gamma_P+[0,1].$$Clearly, $S(\delta)$ is a bounded subset of $S$.

Now we fix any positive smooth function $\Phi$ which vanishes outside $[0,1]$ such that its Fourier transform $\varphi=\hat \Phi$ satisfies $\varphi(0)=1$ and
\begin{equation}\label{ph}
\sup_{x\in\R}(1+x^4)|\varphi (x)|<\infty.
\end{equation}

Set $$
 H(t):=(\Phi\ast \sum_{\gamma\in \Gamma_P}\delta_\gamma)(t)= \sum_{\gamma\in \Gamma_P} \Phi(t-\gamma).
$$Then the support of $H$ belongs to $S(\delta)$ and its Fourier transform is given by
$$
h(x):=\hat H(x)=P(x)\varphi(x).
$$Clealry, $h(0)=1.$

When $\epsilon$  is sufficiently small, from (\ref{p}) and (\ref{ph}) we get
$$
|h(x)|<\frac{\epsilon}{1+x^2}, \mbox{ for all } |x|>\delta,
$$where $\delta$ is the separation constant of $\La$.
Using this estimate and assuming that $\epsilon $ is sufficiently small,   for every $\lambda\in\La$ we get the estimate
$$
\sum_{\mu\in\La,\mu\ne\lambda}|h(\mu-\lambda)|<\sum_{\mu\in\La,\mu\ne\lambda}\frac{\epsilon}{1+(\mu-\lambda)^2}
< 2\sum_{n\in\N}\frac{\epsilon}{1+(\delta n)^2}<\frac{1}{2}.
$$

Set $$M:=\max_{t\in [0,1]}|\Phi(t)|.$$ Then
$$
\int_{S(\delta)}\left|\sum_{\lambda\in\La}c_\lambda e^{i\lambda t}\right|^2\,dt \geq \frac{1}{M}\int_{S(\delta)}\left|\sum_{\lambda\in\La}c_\lambda e^{i\lambda t}\right|^2 H(t)\,dt
$$
$$
=\frac{1}{M}\left(\sum_{\lambda\in\La}|c_\lambda|^2+\sum_{\lambda,\mu\in\La,\lambda\ne\mu}c_\lambda\bar c_\mu h(\lambda-\mu)\right)\geq $$$$\frac{1}{M}\left(\ \sum_{\lambda\in\La}|c_\lambda|^2-\sum_{\lambda,\mu\in\La,\lambda\ne\mu}\frac{|c_\lambda|^2+|c_\mu|^2}{2}|h(\lambda-\mu)|\right)\geq
$$
$$
=\frac{1}{M}\left(\sum_{\lambda\in\La}|c_\lambda|^2-\sum_{\lambda\in\La}|c_\lambda|^2\sum_{\lambda,\mu\in\La,\mu\ne\lambda}|h(\lambda-\mu)|\right)>
\frac{1}{2M}\sum_{\lambda\in\La}|c_\lambda|^2.
$$By Lemma 4, this completes the proof.
\end{proof}



\subsection{Random Spectra do not Admit u.d. Uniqueness Sets}
Here we consider the situation when $S$ is a countable union of unite intervals,  the distances  between the intervals being randomly distributed. More precisely, below we assume that  $S$ and $\Gamma$ are defined in (\ref{gam}) and  (\ref{g}), and that the $\xi_j:=\gamma_{j+1}-\gamma_j-1$  are independent random variables uniformly distributed over the interval $[2,3]$.
With these assumptions, we have

\begin{theorem}\label{tr}
With probably one no u.d. set  $\La$ is  a uniqueness set   for $W_{S}^{(\alpha)}$.
\end{theorem}

\begin{proof}
Theorem 8 follows from the following claim, which is an analogue of the Main Lemma in Sec. 6.1: {\it With probability one, for every fixed $\delta>0$ there is a bounded subset $S(\delta)\subset S$
such that every u.d. set $\La, \delta(\La)\geq\delta,$ is a set of interpolation for $PW_{S(\delta)}$.}

Recall that $$ S=\cup_{j=1}^\infty \{\gamma_j\}+[0,1], \quad \gamma_{j+1}-\gamma_j\in[3,4], \quad j\in\N.$$ It is easy to see that  given any integers $k\geq 1, N\geq 2$ and number $q\in (3,4)$,
 the set
$$
\{\gamma_k, \gamma_k+q,\dots,\gamma_k+Nq\}+[1/4,3/4]
$$belongs to $S$ whenever
$$
|\gamma_{k+j}-(\gamma_k+jq)|<\frac{1}{4},\quad j=1,2,\dots, N.
$$
Recall also that $\gamma_{j+1}-\gamma_{j}$ is uniformly distributed over $[3,4]$. So, the probability that the latter inequalities hold true is positive and independent on $k$.

Now, fix any  $m\in\N$ and $q_1,\dots,q_m\in(3,4)$. By the  the  Borel-Cantelli lemma, one can see that  with probability one there are integers $k_1,\dots,k_m$  such that
the finite sequence
 $$
\Gamma^\ast:= \bigcup_{j=1}^{m} \{\gamma_{k_j}, \gamma_{k_j}+q_j,\dots, \gamma_{k_j}+Nq_j\}
$$satisfies
$$
S(\delta):=\Gamma^\ast+[1/4,3/4]\subset S.
$$
Now, choosing $m$ and $N$ sufficiently large, the claim above follows exactly the same way as in the proof of the Main Lemma.
\end{proof}

\section{Remarks}

\subsection{Multi-dimensional Extensions}

All our one-dimensional results above admit multi-dimensional extensions. Here we give a very brief account of these extensions.


The  definitions in Sec. 1   can be extended to the multi-dimensional situation. In particular,
given a set $S\subset\R^p$, the Paley-Wiener space $PW_S$ consists  of the ($p$-dimensional) inverse Fourier transforms of the $L^2(\R^p)$-functions which vanish a.e. outside $S$.
A  set $\La\subset\R^p$ is uniformly discrete (u.d.), if the infimal distance between its different elements is positive. A u.d. set $\La$ possesses a uniform density $D(\La)$ if
$$
\mbox{Card} (\La\cap ([0,r]^p+s))=r^p D(\La)+o(r^p)\mbox{ uniformly on } s \mbox{ as } r\to\infty.
$$Here  $s=(s_1,\dots,s_p)\in \R^p$ and $$[0,r]^p+s=\{x=(x_1,\dots,x_p)\in\R^p: s_j\leq x_j\leq s_j+r, j=1,\dots,p\}.$$

We will  denote by $|S|$ the $p$-dimensional measure of a set $S\subset\R^p$.



Denote by $S_a$, where $a$ is a positive number, the "projection" of a set $S\subset\R^p$ onto the cube $[0,a]^p$:
$$
S_a:=(S+a\Z^p)\cap [0,a]^p.
$$

We will now formulate  a multi-dimensional analogue of Theorem 1: {\it
Suppose  that $s_1,\dots,s_p$ are real numbers linearly independent over the set of integers. Then the set
$$\La:=\{ m_1+s_12^{-|m_1|-...-|m_p|},\dots, m_p+s_p2^{-|m_1|-...-|m_p|}, (m_1,\dots,m_p)\in\Z^p\}$$
is a uniqueness set for $PW_S$, for every bounded set $S\subset\R^n$ satisfying $|S_1|<1$.}

\medskip
The proof of this result goes on the same lines as the proof of Theorem 2 in \cite{u}.

Choosing the numbers $s_j$ small, one can make the set $\La$ in the above result  an arbitrarily small perturbation of the lattice $\Z^p$.

Also Theorem 2, as noted in \cite{ou}, admits an extension to several dimensions: {\it  For every set $S\subset\R^p$ of finite measure there is  a   u.d. set $\La\subset\R^p, D(\La)= |S|,$ which is a uniqueness set for $PW_S$. }

\medskip

In order to get  multi-dimensional versions of Theorems 3 and 4,  one may introduce  multi-dimensional analogues of the spaces  $W^{(\alpha)}_S$ and $Y_S$ as follows:
The space $W^{(\alpha)}(\R^p), \alpha>p/2,$ consists of functions $f$ defined on $\R^p$ which are the Fourier transform of functions $F$ vanishing outside $S$ and satisfying
$$
\|F\|^2:=\int_{\R^p}(1+|t|^{2\alpha})|F(t)|^2\,dt<\infty,
$$where $|t|^2=t_1^2+\dots+t_p^2$ and $dt=dt_1\cdot\dots\cdot dt_p.$

The space $Y_S(\R^p)$ consists of continuous functions $f$ satisfying
$$\sup_{x\in\R^p}(1 + |x|^{2p})|f(x)| < \infty,   \quad |x|^2:=x_1^2+\dots+x_p^2,$$
and  such that the Fourier transform   $\hat f$ vanishes outside $S$.
 One may check that   both Lemmas 1 and 2 above admit  multi-dimensional extensions. This allows to get multi-dimensional analogues of Theorems 3 and 4: {\it Assume $S\subset\R^p$
 is such that $|S_a|<a^p$, for some $a>0$. Then the spaces $W^{(\alpha)}_S(\R^p), \alpha>p/2,$ and $Y_S(\R^p)$ admit a u.d. uniqueness set}.

\medskip
 One may also  check that the $p$-dimensional versions of Theorems 5--8 hold true.

\subsection{Questions}


We leave open several problems which might be  of certain interest.

\medskip\noindent 1. In connection with Theorems 1 and 2, one may ask:  {\it  Does there exist a u.d. set $\La , D(\La)=1,$
            which is a uniqueness set for $PW_S$, for every set $S\subset\R,|S| < 1$?}

\medskip\noindent 2. The following question arises in connection with Theorems 3 and 4:
                  {\it Let $S\subset\R$ be a set with periodic weak gaps.            Does the space $PW_S\cap  C(\R)$ admit a u.d.  uniqueness set? }

\medskip\noindent 3. It also seems an interesting question if Theorem 2 remains true for the Fourier transforms of integrable functions:
 {\it Let $S\subset\R$ be a set of finite measure. Is it true that the space
$$
              \widehat{L_S} := \{f=\hat F: F\in L^1(\R), F=0 \mbox{ a.e. outside } S\}
$$
           admits a u.d. uniqueness set?}

\medskip
  Theorem 5 implies that the answer is "yes" whenever   $S_a$  is not dense on $[0,a]$, for some $a.$

\medskip

The first author thanks the Israel Science Foundation for  partial support.

We thank also the CIRM center at Trento University for the kind hospitality during our two weeks RiP stay.

\bibliographystyle{amsalpha}

\end{large}

\end{document}